\documentclass[a4paper,12pt,reqno]{amsart}

\usepackage{amsmath}
\usepackage{amssymb}
\usepackage{amsfonts}
\usepackage{graphicx}
\usepackage[colorlinks]{hyperref}
\renewcommand\eqref[1]{(\ref{#1})} 

\graphicspath{ {images/} }
\title[Representation formulae for the higher-order Steklov inequalities]{Representation formulae for the higher-order Steklov and $L^{2^{m}}$-Friedrichs inequalities}

\author[T. Ozawa]{Tohru Ozawa}
\address{
	Tohru Ozawa:
	\endgraf
	Department of Applied Physics
	\endgraf
	Waseda University
	\endgraf
	Tokyo 169-8555
	\endgraf
	Japan
	\endgraf
	{\it E-mail address} {\rm txozawa@waseda.jp}
}

\author[D. Suragan]{Durvudkhan Suragan}
\address{
	Durvudkhan Suragan:
	\endgraf
	Department of Mathematics
	\endgraf
 Nazarbayev University
	\endgraf
	53 Kabanbay Batyr Ave, Astana 010000
	\endgraf
	Kazakhstan
	\endgraf
	{\it E-mail address} {\rm durvudkhan.suragan@nu.edu.kz}}


\thanks{The authors were supported in parts by the Nazarbayev University program 091019CRP2120 and the Nazarbayev University grant 240919FD3901. No new data was collected or generated during the course of this research.}

\newtheoremstyle{theorem}
{10pt}          
{10pt}  
{\sl}  
{\parindent}     
{\bf}  
{. }    
{ }    
{}     
\theoremstyle{theorem}

\numberwithin{equation}{section}
\theoremstyle{plain}
\newtheorem{thm}{Theorem}[section]

\theoremstyle{definition}

\newtheorem{rem}[thm]{Remark}

\newtheoremstyle{defi}
{10pt}          
{10pt}  
{\rm}  
{\parindent}     
{\bf}  
{. }    
{ }    
{}     
\theoremstyle{defi}

\begin{document}
		\begin{abstract}
		In this paper,  we obtain remainder term representation formulae for the higher-order Steklov inequality for vector fields which imply short and direct proofs of the sharp (classical) Steklov inequalities. The obtained results directly imply sharp Steklov type inequalities for some vector fields satisfying H\"ormander's condition, for example. We also give  representation formulae for the $L^{2^{m}}$-Friedrichs inequalities for vector fields.  
	\end{abstract}
	\maketitle
	
\section{Introduction}
In a smooth bounded set $\Omega\subset \mathbb{R}^{n}$, the  Rayleigh quotient for the Laplace operator in $H^{1}_{0}(\Omega)$ is minimized by the ground state with the first eigenvalue $\lambda$ of the minus Dirichlet Laplacian. It directly implies the following (classical) Steklov inequality \cite{Steklov}:
\begin{equation}\label{poincareL2}
\|u\|_{L^{2}(\Omega)} \leq \lambda^{-1/2}\|\nabla u\|_{L^{2}(\Omega)},
\end{equation}
for all  $u\in H^{1}_{0}(\Omega)$.
Furthermore, here the constant $\lambda^{-1/2}$ is sharp.

Let $\Omega\subset \mathbb{R}^{n}$ be a set, bounded at least in one direction. Then there exists a constant $C$, depending only on $\Omega$ and $p$, so that, 
\begin{equation}\label{poincareLp}
\|u\|_{L^{p}(\Omega)} \leq C\|\nabla u\|_{L^{p}(\Omega)},\quad 1<p<\infty,
\end{equation}
for any function $u$ of the Sobolev space $W^{1,p}(\Omega)$ of zero-trace functions.
Inequality \eqref{poincareLp} is called   $L^{p}$-Friedrichs inequality or shortly it can be also called Friedrichs inequality.  

The Steklov inequality is a very important tool proving the existence or/and nonexistence (blow-up) of the solution of partial differential equations and in finite element error estimates. There is a vast number of publications on the Steklov and Friedrichs type inequalities and related subjects (see, e.g. \cite{AW}, \cite{FV}, \cite{KN}, \cite{OS},  and \cite{RS_book} as well as references therein).

The sharp constant in the Steklov inequality is sometimes known as the Steklov constant for $\Omega$. Characterization of the Steklov constant (and its existence) and characterization of nontrivial extremizers (and their existence), in general, very hard tasks that depend upon the value of $p$ and the geometry of $\Omega$. However, if one obtains a representation formula (sharp remainder term formula) for the Steklov inequality, then it follows the proof of the Steklov inequality, 
characterization of the best constant and its existence as well as characterization of nontrivial extremizers and their existence.

Thus, in the present paper, our aim is to obtain representation formulae for the remainder terms for both the higher order version of the Steklov inequality and $L^{2^{m}}$-Friedrichs inequality for general real vector fields. We continue our already started research in this direction \cite{OS19a} and \cite{OS19b}.

As particular cases, the obtained results imply the exact missing term of Steklov inequality \eqref{poincareL2} in the Euclidean case, for example. 

Let $M$ be a smooth $n$-dimension  manifold of a volume form $d\nu$. Let $\{X_{k}\}_{k=1}^{N}$, $N\leq n$, be a family of (smooth) real vector fields on $M$. We denote by $ X$ their gradient 

\begin{equation}\label{EQ:X}
X:=(X_{1},\ldots,X_{N})
\end{equation}
and by $\mathcal L$ corresponding sum of squares operator
\begin{equation}\label{EQ:L}
\mathcal{L}:=\sum_{k=1}^{N}X_{k}^{2}.
\end{equation}
Operators in this form have been much studied and today's literature on the subject is quite large.
For example, it is well-known from H\"ormander's fundamental work \cite{Hor} that if  (the commutators of) the vector fields $\{X_{k}\}_{k=1}^{N}$ generate a Lie algebra, then the sum of squares operator $\mathcal L$ is locally hypoelliptic. Such operators and related estimates have been also studied under weaker (general) conditions or without the hypoellipticity property.

Thus, in this paper,  we obtain remainder terms for the higher order Steklov inequality for the operator $\mathcal{L}$ which imply short and direct proofs of the sharp Steklov inequalities for the (classical) Laplacian. As direct consequences, for instance, the obtained results directly imply sharp Steklov type inequalities for vector fieds satisfying H\"ormander's condition.
We also extend our idea to the $L^{2^{m}}$-Friedrichs inequalities for vector fields. 

Surprisingly, the methods of proofs are just the usual functional analysis arguments by iteration and induction, in addition to few elementary techniques from algebriac relations of vector fields.
The advantage of this method for the sharp remainder term of the Steklov type inequalities is that it allows to treat the case of a general domain with a minimal regularity.

The authors would like to thank Professor Ari Laptev and Professor Grigori Rozenblum for encouragement to complete this paper.

	\section{Representation formula for the higher order Steklov inequality}
	
	\begin{thm}\label{main}
		Let $\Omega \subset M$. Let $\varphi>0$ be a strictly positive eigenfunction of $-\mathcal{L}$ with an eigenvalue $\lambda$, that is, $-\mathcal{L}\varphi= \lambda \varphi$ on $\Omega$. Then for any $u\in C_{0}^{\infty}(\Omega)$
	  we have
\begin{multline}\label{maineven}
\left|X^{2m} u\right|^{2}-\lambda^{2 m} |u|^{2}=\sum_{j=0}^{m-1}\lambda^{2(m-1-j)} \left(\left|\mathcal{L}^{j+1} u+\lambda \mathcal{L}^{j} u\right|^{2}+2 \lambda\left|X \mathcal{L}^{j} u-\frac{X \varphi}{\varphi} \mathcal{L}^{j} u\right|^{2}\right)
\\
+\sum_{j=0}^{m-1}2 \lambda^{2(m-1-j)+1} X\cdot\left(\frac{X \varphi}{\varphi}|\mathcal{L}^{j} u|^{2}-\mathcal{L}^{j} uX\mathcal{L}^{j} u\right),
\end{multline}

where $m=1,2,\ldots,$ and

\begin{multline}\label{mainodd}
\left|X^{2m+1} u\right|^{2}-\lambda^{2 m+1} |u|^{2}=\left|X \mathcal{L}^{m} u-\frac{X \varphi}{\varphi} \mathcal{L}^{m} u\right|^{2}
\\
+\sum_{j=0}^{m-1} \lambda^{2(m-j)-1}\left(\left|\mathcal{L}^{j+1} u+\lambda \mathcal{L}^{j} u\right|^{2}+2 \lambda\left|X \mathcal{L}^{j} u-\frac{X \varphi}{\varphi} \mathcal{L}^{j} u\right|^{2}\right)
\\
+2 \sum_{j=0}^{m-1} \lambda^{2(m-j)} X\cdot\left(\frac{X \varphi}{\varphi}|\mathcal{L}^{j} u|^{2}-\mathcal{L}^{j} uX\mathcal{L}^{j} u\right)+X\cdot\left(\frac{X \varphi}{\varphi}\left(\mathcal{L}^{m} u\right)^{2}\right),
\end{multline}
where $m=0,1,2,\ldots.$
		\end{thm}

	Theorem \ref{main} has consequences in several settings when the assumption is satisfied, most notably, on stratified Lie groups, say, on the Heisenberg group, as well as for vector fields on $\mathbb R^{n}$ satisfying the H\"ormander commutator condition of different steps (c.f., e.g. \cite{CLL}, \cite{FLW} and \cite{Ruzhansky-Suragan:JDE}). For example, in the Euclidean case, Theorem \ref{main} directly implies the sharp missing (remainder) term formula for the higher order Steklov inequality, which gives the sharp Steklov inequality for the polyharmonic operator $\Delta^{m}$.

\begin{thm}\label{main2}
	Let $\Omega \subset \mathbb{R}^{n}$ be a connected domain, for which the divergence theorem is true.
	We have the following remainder of the higher order Steklov inequality 
	
	\begin{multline}\label{mainevenLap}
	\int_{\Omega}\left|\nabla^{2m} u\right|^{2}dx-\lambda_{1}^{2 m} \int_{\Omega}|u|^{2}dx\\=\sum_{j=0}^{m-1}\lambda_{1}^{2(m-1-j)} \left(\int_{\Omega}\left|\Delta^{j+1} u+\lambda_{1} \Delta^{j} u\right|^{2}dx+2 \lambda_{1}\int_{\Omega}\left|\nabla \Delta^{j} u-\frac{\nabla u_{1}}{u_{1}} \Delta^{j} u\right|^{2}dx\right)\geq 0,
	\end{multline}
	where $m=1,2,\dots,$ and
	\begin{multline}\label{mainoddLap}
	\int_{\Omega}\left|\nabla^{2m+1} u\right|^{2}dx-\lambda_{1}^{2 m+1} 	\int_{\Omega}|u|^{2}dx=	\int_{\Omega}\left|\nabla \Delta^{m} u-\frac{\nabla u_{1}}{u_{1}} \Delta^{m} u\right|^{2}dx
	\\
	+\sum_{j=0}^{m-1} \lambda_{1}^{2(m-j)-1}\left(	\int_{\Omega}\left|\Delta^{j+1} u+\lambda_{1} \Delta^{j} u\right|^{2}dx+2 \lambda_{1}	\int_{\Omega}\left|\nabla \Delta^{j} u-\frac{\nabla u_{1}}{u_{1}} \Delta^{j} u\right|^{2}dx\right)\geq 0,
	\end{multline}
		where $m=0,1,\dots,$ for all $u\in C_{0}^{\infty}(\Omega).$ Here $u_{1}$ is the ground state of the (minus) Dirichlet Laplacian in $\Omega$ and $\lambda_{1}$ is the corresponding eigenvalue. The equality cases hold if and only if $u$ is proportional to $u_{1}$.
\end{thm}

\begin{proof}[Proof of Theorem \ref{main2}]
	In connected domains, it is known that the (minus) Dirichlet Laplacian satisfies the so-called Beurling-Deny condition \cite[Appendix 1 to Section XIII.12]{ReedSimon}. That is, the semigroup generated by the operator is positive. Therefore, in Theorem \ref{main} we can set $\lambda=\lambda_{1}>0$ and $\varphi=u_{1}>0$ in $\Omega$.
	We also have 
	$\mathcal L=\Delta$ and $X=\nabla$ in the Euclidean case. Thus, by integrating over $\Omega$ both sides of \eqref{maineven} and \eqref{mainodd} and using the divergence theorem as well as the fact that $u$ vanishes on the boundary of $\Omega$, we obtain \eqref{mainevenLap} and \eqref{mainoddLap}.
	Since $\lambda_{1}>0$ we have the equality case in \eqref{mainevenLap} if and only if 
	$$0=\left|\Delta^{j+1} u+\lambda_{1} \Delta^{j} u\right|^{2},\, j=0,\ldots, m-1,$$
	and
	$$0=\left|\nabla \Delta^{j} u-\frac{\nabla u_{1}}{u_{1}} \Delta^{j} u\right|^{2}=\left|\nabla  \left( \frac{\Delta^{j} u}{u_{1}}\right)\right|^{2}u_{1}^{2},\, j=0,\ldots, m-1,$$ that is, $\frac{u}{u_{1}}=const$.
	This means the equality case in \eqref{mainevenLap} holds if and only if 
	$u$ is proportial to $u_{1}$. 
	
	Similarly, since $\lambda_{1}>0$ we have the equality case  in \eqref{mainoddLap} if and only if 
	$$0=\left|\Delta^{j+1} u+\lambda_{1} \Delta^{j} u\right|^{2},\quad j=0,\ldots, m-1,$$
	and
	$$0=\left|\nabla \Delta^{j} u-\frac{\nabla u_{1}}{u_{1}} \Delta^{j} u\right|^{2}=\left|\nabla  \left( \frac{\Delta^{j} u}{u_{1}}\right)\right|^{2}u_{1}^{2},\, j=0,\ldots, m,$$ that is, $\frac{u}{u_{1}}=const$.
	This means the equality case in \eqref{mainoddLap} also holds if and only if 
	$u$ is proportial to $u_{1}$. 
	It completes the proof. 
\end{proof}

		\begin{proof}[Proof of Theorem \ref{main}]
		
To prove equality \eqref{maineven}. First let us check its validity for $m=1$. A straightforward calculation gives

\begin{multline}\label{m=2}
\left|\mathcal{L} u-\frac{\mathcal{L} \varphi}{\varphi} u\right|^{2}=|\mathcal{L} u|^{2}-2 \frac{\mathcal{L} \varphi}{\varphi} u \mathcal{L} u+\left(\frac{\mathcal{L} \varphi}{\varphi}\right)^{2} |u|^{2}
\\
=|\mathcal{L} u|^{2}-\frac{\mathcal{L} \varphi}{\varphi}\left(\mathcal{L}|u|^{2}-2|X u|^{2}\right)+\left(\frac{\mathcal{L} \varphi}{\varphi}\right)^{2} |u|^{2}.
\end{multline}
On the other hand, we have 
$$
2 \frac{\mathcal{L} \varphi}{\varphi}|X u|^{2}=2 \frac{\mathcal{L} \varphi}{\varphi}\left(-\frac{\mathcal{L} \varphi}{\varphi} |u|^{2}+\left|X u-\frac{X \varphi}{\varphi} u\right|^{2}+X\cdot\left(\frac{X \varphi}{\varphi} |u|^{2}\right)\right).
$$
Combining this with \eqref{m=2} and using the assumption $-\mathcal{L}\varphi= \lambda \varphi$, we obtain
$$
\left|\mathcal{L} u-\frac{\mathcal{L} \varphi}{\varphi} u\right|^{2}
$$

$$
=|\mathcal{L} u|^{2}-\left(\frac{\mathcal{L} \varphi}{\varphi}\right)^{2} |u|^{2}+2 \frac{\mathcal{L} \varphi}{\varphi}\left(\left|X u-\frac{X \varphi}{\varphi} u\right|^{2}+X\cdot\left(\frac{X \varphi}{\varphi} |u|^{2}\right)\right)-\frac{\mathcal{L} \varphi}{\varphi}\mathcal{L}|u|^{2}
$$

$$
=|\mathcal{L} u|^{2}-\lambda^{2} |u|^{2}-2 \lambda\left(\left|X u-\frac{X \varphi}{\varphi} u\right|^{2}+X\cdot\left(\frac{X \varphi}{\varphi} |u|^{2}\right)\right)+\lambda\mathcal{L}|u|^{2}.
$$
Now using the identity 
$$
\lambda \mathcal{L}|u|^{2}=2 \lambda X\cdot\left(uXu\right),
$$
we get
$$|\mathcal{L} u|^{2}-\lambda^{2} |u|^{2}=\left|\mathcal{L} u+ \lambda u\right|^{2}+2 \lambda\left(\left|X u-\frac{X \varphi}{\varphi} u\right|^{2}+X\cdot\left(\frac{X \varphi}{\varphi} |u|^{2}-uXu\right)\right).$$
Now we need to complete the inductive step
$m \Rightarrow m+1.$ Thus, we compute 

$$
\left|\mathcal{L}^{m +1} u\right|^{2}-\lambda^{2(m+1)} |u|^{2}
$$

$$
=\left|\mathcal{L} \mathcal{L}^{m} u\right|^{2}-\lambda^{2}\left|\mathcal{L}^{m} u\right|^{2}+\lambda^{2}\left(\left|\mathcal{L}^{m} u\right|^{2}-\lambda^{2 m} |u|^{2}\right)
$$

$$
=\left|\mathcal{L}^{m+1} u+\lambda \mathcal{L}^{m} u\right|^{2}+2 \lambda\left|X \mathcal{L}^{m} u-\frac{X \varphi}{\varphi} \mathcal{L}^{m} u\right|^{2}
$$

$$
+2 \lambda X\cdot\left(\frac{X \varphi}{\varphi}|\mathcal{L}^{m} u|^{2}-\mathcal{L}^{m}u \,X\mathcal{L}^{m} u\right)
$$

$$
+\sum_{j=0}^{m-1} \lambda^{2(m-j)}\left(\left|\mathcal{L}^{j+1} u+\lambda \mathcal{L}^{j} u\right|^{2}+2 \lambda\left|X \mathcal{L}^{j} u-\frac{X \varphi}{\varphi} \mathcal{L}^{j} u\right|^{2}\right)
$$

$$
+2 \sum_{j=0}^{m-1} \lambda^{2(m-j)+1} X\cdot\left(\frac{X \varphi}{\varphi}|\mathcal{L}^{j} u|^{2}-\mathcal{L}^{j} uX\mathcal{L}^{j} u\right)
$$

$$
=\sum_{j=0}^{m} \lambda^{2(m-j)}\left(\left|\mathcal{L}^{j+1} u+\lambda \mathcal{L}^{j} u\right|^{2}+2 \lambda\left|X \mathcal{L}^{j} u-\frac{X \varphi}{\varphi} \mathcal{L}^{j} u\right|^{2}\right)
$$

$$
+2 \sum_{j=0}^{m} \lambda^{2(m-j)+1}  X\cdot\left(\frac{X \varphi}{\varphi}|\mathcal{L}^{j} u|^{2}-\mathcal{L}^{j} uX\mathcal{L}^{j} u\right).
$$
It proves formula \eqref{maineven}.

In order to prove relation \eqref{mainodd}, we provide the following direct computation
\begin{equation}
\label{eq1}
\left|X u-\frac{X \varphi}{\varphi} u \right|^{2}=|X u|^{2}-2\frac{X \varphi}{\varphi} u X u+\frac{|X \varphi|^{2}}{\varphi^{2}}|u|^{2}=
|X u|^{2}-\frac{X \varphi}{\varphi}  X |u|^{2}+\frac{|X \varphi|^{2}}{\varphi^{2}}|u|^{2}.		\end{equation}
We also have
\begin{equation}
\label{eq2} -\frac{X \varphi}{\varphi}  X |u|^{2}
=-X\cdot \left(\frac{X \varphi}{\varphi} |u|^{2} \right) +\frac{\mathcal{L} \varphi}{\varphi}|u|^{2}-\frac{|X \varphi|^{2}}{\varphi^{2}}|u|^{2}.	\end{equation}
Combining \eqref{eq1} and \eqref{eq2}, we obtain
\begin{equation}
\label{eq3}
|X u|^{2}-\lambda |u|^{2}=\left|X u-\frac{X \varphi}{\varphi} u\right|^{2}+X\cdot\left(\frac{X \varphi}{\varphi} |u|^{2}\right).
\end{equation}
That is, it verifies 
\eqref{mainodd} when $m=0.$

Now we apply the scaling 
$$
u \mapsto \mathcal{L}^{m} u
$$
to \eqref{eq3}. Thus, we get 
\begin{equation}
\label{eq4}
\left|X \mathcal{L}^{m} u\right|^{2}=\left|X \mathcal{L}^{m} u-\frac{X \varphi}{\varphi} \mathcal{L}^{m} u\right|^{2}+X\cdot\left(\frac{X \varphi}{\varphi}\left(\mathcal{L}^{m} u\right)^{2}\right)+\lambda\left(\mathcal{L}^{m} u\right)^{2}.
\end{equation}

Adding $-\lambda^{2 m+1} |u|^{2}$ to both sides of  \eqref{eq4} and using formula \eqref{maineven}, we arrive at 
$$
\left|X \mathcal{L}^{m} u\right|^{2}-\lambda^{2 m+1} |u|^{2}
$$

$$
=\left|X \mathcal{L}^{m} u-\frac{X \varphi}{\varphi} \mathcal{L}^{m} u\right|^{2}+X\cdot\left(\frac{X \varphi}{\varphi}\left(\mathcal{L}^{m} u\right)^{2}\right)
$$

$$
+\lambda\left(\left(\mathcal{L}^{m} u\right)^{2}-\lambda^{2 m} |u|^{2}\right)
$$

$$
=\left|X \mathcal{L}^{m} u-\frac{X \varphi}{\varphi} \mathcal{L}^{m} u\right|^{2}+X\cdot\left(\frac{X \varphi}{\varphi}\left(\mathcal{L}^{m} u\right)^{2}\right)
$$

$$
+\sum_{j=0}^{m-1} \lambda^{2(m-j)-1}\left(\left|\mathcal{L}^{j+1} u+\lambda \mathcal{L}^{j} u\right|^{2}+2 \lambda\left|X \mathcal{L}^{j} u-\frac{X \varphi}{\varphi} \mathcal{L}^{j} u\right|^{2}\right)
$$

$$
+2 \sum_{j=0}^{m-1} \lambda^{2(m-j)}  X\cdot\left(\frac{X \varphi}{\varphi}|\mathcal{L}^{j} u|^{2}-\mathcal{L}^{j} uX\mathcal{L}^{j} u\right).
$$
It completes the proof. 

	\end{proof}

	\section{Representation formula for the  $L^{2^{m}}$-Friedrichs inequality}
Note that here  $u\in C^{1}(\Omega)$ means that $Xu\in C(\Omega)$. 
	\begin{thm}\label{mainL2m} Let $\Omega\subset M$. Let $m$ be a nonnegative integer.
		For all $u\in C^{1}(\Omega)$ and $\varphi\in C^{2}(\Omega)$ with $\varphi>0$, we have 
		\begin{multline}\label{poincareLpm}
		|X u|^{p_{m}}+\left(\frac{\mathcal{L} \varphi}{\varphi}+\sigma_{m}\right) u^{p_{m}}
		=\sum_{j=1}^{m-1}\left|| X\left(u^{p_{m-j-1}}\right)|^{p_{j}}-2^{p_{j}-1} u^{p_{m-1}}\right|^{2}
	\\
		+\left|X\left(u^{p_{m-1}}\right)-\frac{X \varphi}{\varphi} u^{p_{m-1}}\right|^{2}+X\cdot\left(\frac{X \varphi}{\varphi} u^{p_{m}}\right).
		\end{multline}
Here $p_{m}=2^{m},\, m \geqslant 0,$
	and	
$\sigma_{m}=\frac{1}{4} \sum_{j=1}^{m-1} 4^{p_{j}},\, m \geqslant 1.$
\end{thm}

	\begin{rem}
We have 
$$
4^{p_{m-1}-1} \leqslant \sigma_{m} \leqslant \frac{m-1}{4}\, 4^{p_{m-1}},
$$
that is, 
$$
4^{\frac{p_{m-1}-1}{p_{m}}} \leqslant \sigma_{m}^{\frac{1}{p_{m}}}\leqslant\left(\frac{m-1}{4}\right)^{\frac{1}{p_{m}}}  4^{\frac{p_{m-1}}{p_{m}}}.
$$
It follows the following asymtotics on $
\sigma_{m}$:
$$
\lim _{m \rightarrow \infty} \sigma_{m}^{\frac{1}{p_{m}}}=2.
$$
	\end{rem}

\begin{thm}\label{mainLp2}
	Let $\Omega \subset \mathbb{R}^{n}$ be a connected domain, for which the divergence theorem is true. 
	We have the remainder of the $L^{2^{m}}$-Friedrichs inequality 
	
	\begin{multline}\label{mainLpLap}
	\int_{\Omega}|\nabla u|^{p_{m}}dx-\left(\lambda_{1}-\sigma_{m}\right) \int_{\Omega}|u|^{p_{m}}dx\\=\sum_{j=1}^{m-1} \int_{\Omega}\left|| \nabla\left(u^{p_{m-j-1}}\right)|^{p_{j}}-2^{p_{j}-1} u^{p_{m-1}}\right|^{2}dx+\int_{\Omega}\left|\nabla\left(u^{p_{m-1}}\right)-\frac{\nabla u_{1}}{u_{1}} u^{p_{m-1}}\right|^{2}dx\geq 0
	\end{multline}
	for all $u\in C_{0}^{1}(\Omega).$ Here $\sigma_{m}=\frac{1}{4} \sum_{j=1}^{m-1} 4^{p_{j}},\, m \in \mathbb{N},\, p_{j}=2^{j},$ $u_{1}$ is the ground state of the minus Laplacian in $\Omega$ and $\lambda_{1}$ is the corresponding eigenvalue. 
\end{thm}

Of course, the difference $\lambda_{1}-\sigma_{m}$ can be negative for some domains.  It means that \eqref{mainLpLap} does not imply the $L^{2^{m}}$-Friedrichs inequality, in general. 

Note that, as usual, when $m=1$ the term with the sigma notation in \eqref{mainLpLap}  disappears since the lower index is greater than the upper one.
\begin{proof}[Proof of Theorem \ref{mainLp2}]
As in the proof of Theorem \ref{main2}, in Theorem \ref{mainL2m} we set $\lambda=\lambda_{1}>0$ and $\varphi=u_{1}>0$ in $\Omega$, and $X=\nabla$. Thus, by integrating over $\Omega$ both sides of \eqref{mainL2m}  and using the divergence theorem as well as the fact that $u$ vanishes on the boundary of $\Omega$, we arrive at the desired result. 
\end{proof}

	\begin{proof}[Proof of Theorem \ref{mainL2m}]
When $m=1$, we have $p_{1}=2$, $\sigma_{1}=0$,
and

$$
|X u|^{2}+\frac{\mathcal{L} \varphi}{\varphi} u^{2}=\left|X u-\frac{X \varphi}{\varphi} u\right|^{2}+X\cdot\left(\frac{X \varphi}{\varphi} u\right).
$$
When $m=2$, we have $p_{2}=4$, $\sigma_{2}=\frac{1}{4}\, 4^{p_{1}}=4$, and
$$
|X u|^{4}+\left(\frac{\mathcal{L} \varphi}{\varphi}+4\right) u^{4}=\left.|| X u\right|^{2}-\left.2 |u|^{2}\right|^{2}+
$$
$$
+\left|X\left(|u|^{2}\right)-\frac{X \varphi}{\varphi} |u|^{2}\right|^{2}+X\cdot\left(\frac{X \varphi}{\varphi} u^{4}\right).
$$
When $m=3$, we have $p_{3}=8$, $\sigma_{3}=\sigma_{2}+\frac{1}{4}\, 4^{p_{2}}=4+4^{3}=68$, and 
$$
|X u|^{8}+\left(\frac{\mathcal{L} \varphi}{\varphi}+68\right) u^{8}=\left|| X u |^{4}-8 u^{4}\right|^{2}+\left|| X\left(|u|^{2}\right)|^{2}-2 u^{4}\right|^{2}
$$

$$
+\left|X\left(u^{4}\right)-\frac{X \varphi}{\varphi} u^{4}\right|^{2}+X\cdot\left(\frac{X \varphi}{\varphi} u^{8}\right).
$$
To complete the induction  we start with

$$\left|| X u |^{p_{m}}-2^{p_{m}-1} u^{p_{m}}\right|^{2}
=|X u|^{2 p_{m}}-2^{p_{m}} u^{p_{m}}|X u|^{p_{m}}+2^{2 p_{m}-2}  u^{2 p_{m}}
$$
$$
=|X u|^{p_{m+1}}-2^{p_{m}} u^{p_{m}}|X u|^{p_{m}}+\frac{1}{4}\, 4^{p_{m}} u^{p_{m+1}}.
$$
We replace \(u\) by \(|u|^{2}\) in \eqref{poincareLpm} to obtain

$$
\left|X\left(|u|^{2}\right)\right|^{p_{m}}+\left(\frac{\mathcal{L} \varphi}{\varphi}+\sigma_{m}\right)\left(|u|^{2}\right)^{p_{m}}
$$

$$
=\sum_{j=1}^{m-1} \left|\left|X\left(|u|^{2}\right)^{p_{m-j-1}}\right|^{p_{j}}-2^{p_{j}-1}\left(|u|^{2}\right)^{p_{m-1}}\right|^{2}
$$

$$
+\left|X\left(|u|^{2}\right)^{p_{m-1}} -\frac{X \varphi}{\varphi}\left(|u|^{2}\right)^{p_{m-1}}\right|^{2}+X\cdot\left(\frac{X \varphi}{\varphi}\left(|u|^{2}\right)^{p_{m}}\right),
$$
where 

$$
\left|X\left(|u|^{2}\right)\right|^{p_{m}}+\left(\frac{\mathcal{L} \varphi}{\varphi}+\sigma_{m}\right)\left(|u|^{2}\right)^{p_{m}}=2^{p_{m}} u^{p_{m}}\left|X u \right|^{p_{m}}+\left(\frac{\mathcal{L} \varphi}{\varphi}+\sigma_{m}\right) u^{p_{m+1}}
$$
and 
$$
\sum_{j=1}^{m-1} \left|\left|X\left(|u|^{2}\right)^{p_{m-j-1}}\right|^{p_{j}}-2^{p_{j}-1}\left(|u|^{2}\right)^{p_{m-1}}\right|^{2}
$$

$$
+\left|X\left(|u|^{2}\right)^{p_{m-1}} -\frac{X \varphi}{\varphi}\left(|u|^{2}\right)^{p_{m-1}}\right|^{2}+X\cdot\left(\frac{X \varphi}{\varphi}\left(|u|^{2}\right)^{p_{m}}\right)
$$
$$
=\sum_{j=1}^{m-1} \left| |X\left(u^{p_{m-j}}\right)|^{p_{j}}-2^{p_{j}-1} u^{p_{m}}\right|^{2}
$$

$$
+\left|X\left(u^{p_{m}}\right)-\frac{X \varphi}{\varphi} u^{p_{m}}\right|^{2}+X\cdot\left(\frac{X \varphi}{\varphi} u^{p_{m+1}}\right).
$$
Therefore, we have
$$|X u|^{p_{m+1}}
=2^{p_{m}} u^{p_{m}}\left|X u\right|^{p_{m}}+\left|| X u|^{p_{m}}-2^{p_{m}-1} u^{p_{m}}\right|^{2}-\frac{1}{4}\, 4^{p_{m}} u^{p_{m+1}}
$$
$$
=-\left(\frac{\mathcal{L} \varphi}{\varphi}+\sigma_{m}+\frac{1}{4}\, 4^{p_{m}}\right) u^{p_{m+1}}+\left| |X u|^{p_{m}}-2^{p_{m}-1} u^{p_{m}}\right|^{2}
$$
$$
+\sum_{j=1}^{m-1} \left| |X\left(u^{p_{m-j}}\right)|^{p_{j}}-2^{p_{j}-1} u^{p_{m}}\right|^{2}
$$

$$
+\left|X\left(u^{p_{m}}\right)-\frac{X \varphi}{\varphi} u^{p_{m}}\right|^{2}+X\cdot\left(\frac{X \varphi}{\varphi} u^{p_{m+1}}\right).
$$
With $$
\sigma_{m}+\frac{1}{4}\, 4^{p_{m}}=\sigma_{m+1},
$$
it completes the induction. 
	\end{proof}

\end{document}